\documentclass{amsart}
\usepackage{amsmath}
\usepackage{color}
\usepackage{graphicx}
\usepackage{amsfonts}
\usepackage{amssymb}
\usepackage{graphicx}
\usepackage{amsthm}
\usepackage[usenames,dvipsnames]{xcolor}
\newtheorem{thm}{Theorem}
\newtheorem*{thm*}{Theorem}
\newtheorem*{thm:bTh}{Theorem \ref{thm:bTh}}
\newtheorem*{thm:big1}{Theorem \ref{thm:big1}}

\newtheorem{lemma}[thm]{Lemma}

\theoremstyle{definition}
\newtheorem{definition}[thm]{Definition}

\theoremstyle{remark}
\newtheorem{remark}[thm]{Remark}

\newtheorem{theorem}[thm]{Theorem}
\numberwithin{equation}{section}

\def\theenumi{\arabic{enumi}}

\def\theenumii{\alph{enumii}}
\def\p@enumii{\theenumi.}

\def\theenumiii{\arabic{enumiii}}
\def\p@enumiii{(\theenumi)(\theenumii)}

\def\p@enumiv{\p@enumiii.\theenumiii}

\newcommand{\N}{\mathbb{N}}

\begin{document}

\title[Generalizing the Futurama Theorem]{Generalizing the Futurama Theorem}
\author{J. Elder} 
\address{Department of Mathematics. Arizona State University. Tempe, AZ 85287.}
\email{Jennifer.E.Elder@asu.edu}

\author{O. Vega} 
\address{Department of Mathematics. California State University, Fresno. Fresno, CA 93740.}
\email{ovega@csufresno.edu}

\thanks{The first author thanks the support of the Division of Graduate Studies at California State University, Fresno in the form of the Robert \& Norma Craig Graduate Scholarship and also for travel support.  She would also like to thank the College of Science and Mathematics at California State University, Fresno for travel support.}

\begin{abstract}
The 2010 episode of \textit{Futurama} titled \textit{The Prisoner of Benda} centers around a machine that swaps the brains of any two people who use it. The problem is, once two people use the machine to swap brains with each other, they cannot swap back. 

In this article, we present a new proof of this theorem and also a generalization of it to what would happen if, instead, the machine swapped cyclically the brains of $p$ characters, where $p$ is prime.
\end{abstract} 

\maketitle

\section{Introduction}

In the 2010 Writers Guild Award-winning episode of \textit{Futurama}, \textit{The Prisoner of Benda} \cite{WGA}, Professor Farnsworth and Amy build a machine that can swap the brains of any two people. The two use the machine to swap brains with each other, but then discover that once two people have swapped with each other, the machine does not swap them back. More characters get involved until the group is thoroughly mixed up, and they start looking for ways to return to their own heads.

Clearly, the problem of undoing what the machine has done may be studied using permutations; each brain swap can be described by a transposition in $S_n$, where $n$ is the number of characters involved in the brain-swapping.  Hence, the problem consists in writing the inverse of a permutation as a product of transpositions that (1) are all distinct, and (2) were not already used in constructing the original permutation. 

The solution given in the episode relies on involving two additional people; these two characters were not involved in any of the original swaps, and so have no restrictions on who they can swap with now. For example, using the additional people $x$ and $y$, we have
\[
(1~2)^{-1}  =  (x~y)(2~x)(1~y)(2~y)(1~x)
\]

Note that the transpositions used in the product above are all distinct, and that none of them is an element of $S_2$, which is where the original element $(1~2)$ was taken from. This can be done in general for any permutation, and leads to the following theorem.

\begin{theorem}[Keeler, 2010]\label{cyclesthm}
Let $n\in \N$, $n\geq 2$. The inverse of any permutation in $S_n$ can be written as a product of distinct transpositions in $S_{n+2} \setminus S_n$.
\end{theorem}

Although Ken Keeler (writer and executive producer for Futurama at the time) never published the proof of his theorem, the proof given in the episode as well as a write up of the idea may be found in \cite{FT}. This was, most probably, the first time in which a theorem was proved for the purpose to advance the narrative of a TV show, and thus it was all over the internet \cite{Gz, AMS, IS} and was also mentioned in books, such as Simon Singh's \emph{The Simpsons and Their Mathematical Secrets} \cite{SS}. Of course, this aroused the interest of mathematicians (the authors of this article included), and so in 2014, Evans, Huang, and Nguyen \cite{EHN}  proved that Keeler's solution to the problem is optimal in the sense that it uses the minimal number of cycles and the minimal number of additional elements. In that article, they also gave necessary and sufficient conditions on $m$ and $n$ for the identity permutation to be expressible as a product of $m$ distinct transpositions in $S_n$.

Although in the title of this article we used \emph{The Futurama Theorem}, for the remainder of  this work we will refer to Theorem \ref{cyclesthm} as \emph{Keller's Theorem}.

\section{Reproving Keeler's Theorem} 

We want to reprove Keeler's Theorem using a different, and possibly less optimal, method. The goal is to present a method that can be generalized to larger cycles in $S_n$.  As in Keeler's proof, we will need additional elements that could not have been used in the permutation we wish to work with. 

As expected, we will first study how to write the inverse of a cycle using `new' transpositions to later putting all this together using that every permutation can be written as the product of disjoint cycles.

\begin{definition}
Fix $n\in \N$ and let $x, y\in \N$ be such that $x, y \gg n$. We consider the set that we will use to define the elements of $S_{n+2}$ to be $\{1,2,\ldots , n, x,y\}$. 
\end{definition}

Throughout this paper, we will consider, without loss of generality, the $k$-cycle $(1~2~\cdots ~ k)$ instead of a generic $k$-cycle. We start by setting some notation.

\begin{definition} \label{deltadef}
Let $m\in \N$ such that $m$ is odd and $m\leq n-2$. And let $\delta_m\in S_{n+2}$ be defined as 
\[
\delta_m = \prod_{i=1 \atop {i \ \text{odd}}}^{m}(i+1~y)(i+2~y)(i~x)(i+1~x).
\]
\end{definition}

Notice that $\delta_m$, for any $1\leq m$, will not contain any repeated transpositions. In fact, every number $1\leq j \leq m$ will appear at most twice in the product: once in $(j~x)$ and once in $(j~y)$. Also, we can see that $(x~y)$ is not in the product. 

\begin{lemma}\label{lem1}
Let $k\geq 7$ be odd and such that $k\equiv 1 \pmod 3$. If $\sigma$ is a $k$-cycle then  $\sigma^{-1} = \delta_{k-2}$. 
\end{lemma}

\begin{proof} 
We write $k = 1+6t$, for $t\in \N$. We will prove the result by induction on $t$. For $t=1$, we get:
\[
(1~2~3~4~5~6~7)^{-1} =  (6~9)(7~9)(5~8)(6~8)(4~9)(5~9)(3~8)(4~8)(2~9)(3~9)(1~8)(2~8) = \delta_{5}
\]

We now assume the result for all $k$-cycles, where $k=1+6s$ and $s<t$. Now let $\sigma$ be a $j$-cycle, where $j=1+6t$. We let $k=j-6$ and write
\begin{equation}\label{eqn1}
(1~2~\cdots ~j)^{-1} =(j-6~j-5~j-4~j-3~j-2~j-1~j)^{-1}  (1~2~\cdots ~k)^{-1}
\end{equation}

From the induction hypothesis, we get
\[
(1~2~\cdots ~k)^{-1}=\delta_{j-8},
\]
and from the base case we get 
\[
(j-6~j-5~j-4~j-3~j-2~j-1~j)^{-1}  =  \prod_{i=j-6 \atop {i \ \text{odd}}}^{j-2}(i+1~y)(i+2~y)(i~x)(i+1~x).
\]

Putting all this  together in Equation (\ref{eqn1}) yields:
\[
(1~2~\cdots ~j)^{-1}  = \left( \prod_{i=j-6 \atop {i \ \text{odd}}}^{j-2}(i+1~y)(i+2~y)(i~x)(i+1~x) \right)  \delta_{j-8}  = \delta_{j-2},
\]
which is what we wanted to obtain.
\end{proof}

We will use this lemma to prove all remaining cases. First for $k$ odd.

\begin{lemma}\label{lem2}
Let $k$ be an odd number and $\sigma$ be a $k$-cycle.
\begin{enumerate}
\item If $k\geq 5$ and $k\equiv 2 \pmod 3$, then $\sigma^{-1} = (x~y)(k-2~x)(k-1~x)(k~x)\delta_{k-4}$. 
\item If $k\geq 3$ and $k\equiv 0 \pmod 3$, then $\sigma^{-1} = (x~y)(k~x)\delta_{k-2}$. 
\end{enumerate}
\end{lemma}

\begin{proof}
Let $k$ be an odd number and $\sigma$ be a $k$-cycle.
\begin{enumerate}
\item For $k=5$, we see that
\[
(1~2~3~4~5)^{-1}  =  (6~7)(3~6)(4~6)(5~6)(2~7)(3~7)(1~6)(2~6)
\]

For $k>5$ and $k\equiv 2\pmod 3$, we write
\[
(1~2~\ldots k) = (1~2~\ldots k-4) (k-4~k-3~k-2~k-1~k)
\]

Since $k-4$ is  odd and $k-4 \equiv 1\pmod 3$ we use Lemma \ref{lem1} to get the desired result.

\item For $k=3$, we have
\[
(1~2~3)^{-1} = (4~5)(3~4)(2~5)(3~5)(1~4)(2~4),
\]
and when $k>3$ and $k\equiv 0\pmod 3$, we know that $k-2 \equiv 1\pmod 3$, with $k-2$ odd. Hence
\[
(1~2~\ldots k) = (1~2~\ldots k-2) (k-2~k-1~k).
\]
\end{enumerate}
\hspace{.5in}At this point we use Lemma \ref{lem1} to get the desired result.
\end{proof}

Now we will proceed in a similar manner for the cases where $k$ is even.

\begin{lemma}\label{lem4}
Let $k$ be a positive even number, and  $\sigma$ be a $k$-cycle.
\begin{enumerate}
\item If $k\geq 6$  and $k\equiv 0 \pmod 3$, then $\sigma^{-1} =(k~y)(k-1~x)(k~x) \delta_{k-3}$. 
\item If $k\geq 4$ and $k\equiv 1 \pmod 3$, then $\sigma^{-1} = (x~y)(k-1~x)(k~x)\delta_{k-3}$. 
\item If $k\geq 8$  and $k\equiv 2 \pmod 3$, then 
\[
\sigma^{-1} =(k-2~y)(k-1~y)(k~y)(k-3~x)(k-2~x) \delta_{k-5}.
\]
\end{enumerate}
\end{lemma}

\begin{proof}
Let $k$ be even, and  $\sigma$ be a $k$-cycle.
\begin{enumerate}
\item For $k=6$, we see that
\[
(1~2~3~4~5~6)^{-1} =  (6~8)(5~7)(6~7)(4~8)(5~8)(3~7)(4~7)(2~8)(3~8)(1~7)(2~7)
\]

For $k>6$, and $k\equiv 0\pmod 3$, we write 
\[
(1~2~\ldots k) = (1~2~\ldots k-5) (k-5~k-4~k-3~k-2~k-1~k)
\]

Since $k-5$ is odd and $k-5 \equiv 1\pmod 3$ we can use Lemma \ref{lem1} to get the desired result

\item For $k=4$, we have
\[
(1~2~3~4)^{-1} =  (5~6)(3~5)(4~5)(2~6)(3~6)(1~5)(2~5)
\]

and when $k>4$ and $k\equiv 1 \pmod 3$, we know that $k-3 \equiv 1\pmod 3$, $k-4$ is odd. Hence
\[
(1~2~\ldots k) = (1~2~\ldots k-3) (k-3~k-2~k-1~k)
\]
At this point we use Lemma \ref{lem1} to get the desired result.

\item For $k=8$,
\[
\sigma^{-1} =(6~y)(7~y)(k~y)(5~x)(6~x) \delta_{3}
\]

When $k>8$ and $k\equiv 2 \pmod 3$, we write
\[
(1~2~\ldots k) = (1~2~\ldots k-7) (k-7~\ldots~k-1~k)
\]
\end{enumerate}
\hspace{.5in}Since $k-7 \equiv 1\pmod 3$, we use Lemma \ref{lem1} to get the desired result.
\end{proof}

\begin{remark}
Lemmas \ref{lem1}, \ref{lem2}, and \ref{lem4} allow us to write any given $k$-cycle as a product of transpositions. Moreover, such products consist of transpositions that are all distinct from each other, and each one of these $2$-cycles is an element in $S_{n+2} \setminus S_n$.
\end{remark}

Finally, given a permutation $\sigma \in S_n$, we write its inverse as a product of disjoint cycles, $\sigma^{-1} = \sigma_1 \sigma_2 \cdots \sigma_r$, and then we write each $\sigma_i$ as a product of transpositions in $S_{n+2} \setminus S_n$ using Lemmas \ref{lem1}, \ref{lem2}, and \ref{lem4}. It is easy to see that the only transposition that could appear more than once in this product is $(x~y)$. In this case, each $(x~y)$ would appear at the beginning of a product representing a $\sigma_i$ (see Lemmas \ref{lem2} and \ref{lem4}), but this situation is easy to fix as if $\tau$ is a permutation in $S_{n+2} \setminus S_n$ then the conjugation $(x~y)  \tau  (x~y)$ yields a permutation that is equal to $\tau$ where all the $x$s and the $y$s have been exchanged. Keeler's Theorem follows.

\section{Products of larger cycles} 

Now we are interested in learning what would happen if Professor Farnsworth and Amy had created a machine that swapped several brains in a cyclical way. For example, the machine would do
\[
\text{Professor} \longrightarrow \text{Amy}  \longrightarrow \text{Bender}   \longrightarrow \text{Leela}   \longrightarrow \text{Fry}  \longrightarrow \text{Professor} 
\]

The problem of undoing what the machine does in the example above would be easy if the machine did not have any `issues' and worked even when a set of five people sat in the machine for a second time;  we could use that $(1~2~3~4~5)^{-1} = (5~4~3~2~1)$ or that $(1~2~3~4~5)^{-1} = (1~2~3~4~5)^4$ to solve the problem. That would be too easy, there would be no TV network that would broadcast an episode featuring such a simple problem/solution (the existence of the show Scooby-Doo may go against this premise). Hence, to make this interesting, we will assume that the machine does not always work when a set of people who have already sat in the machine sits in it again. Our restriction will be technical, and so we will phrase it explicitly in the statement of our results.

So, for the rest of this article, we will assume that we are working with a machine that swaps cyclically the brains of $p$ people at a time, where $p$ is an odd prime. Note that now, in terms of permutations, our problem will reside in $A_n$, as the brain-swapping machine yields an odd cycle every time it functions.

We will start our study with what will end up being a special case: $p=3$.

\subsection{Products of $3$-cycles} 

It turns out that, in this case, we only need one extra character to undo the mess created by the machine.

\begin{lemma}\label{lem3cycles1}
Let $n\in \N$, $n>2$, and let $\sigma \in A_n$ be an odd cycle. Then, $\sigma^{-1}$ can be written as a product of $3$-cycles
\[
\sigma^{-1} = \tau_1 \tau_2 \cdots \tau_t
\]
where $\tau_i \in A_{n+1} \setminus A_n$,  for all $i=1,2,\ldots, t$, and $\tau_i \notin <\tau_j>$, for all $i\neq j$.
\end{lemma}

\begin{proof}
For $k=3$ we get $(1~2~3) = (x~ 3~ 1)(x~ 1 ~ 2)$, and for every $3> k\leq n$, $k$ odd, and $x \gg n$, consider the product 
\[
(1~2~\cdots k-1 ~ k) = (x~ k ~ 1)(x~ k-2 ~ k-1) \cdots (x~ 3 ~ 4)(x~ 1 ~ 2).
\]

The result is now immediate, as the inverse of any odd cycle is also an odd cycle, and the $3$-cycles used in the products considered all fix different sets of elements.
\end{proof}

\begin{lemma}\label{lem3cycles2}
Let $\alpha, \beta \in S_n$, $n>2$, be two disjoint even cycles. Then, $(\alpha \beta)^{-1}$ can be written as a product of $3$-cycles
\[
(\alpha \beta)^{-1} = \tau_1 \tau_2 \cdots \tau_t,
\]
where $\tau_i \in A_{n+1} \setminus A_n$,  for all $i=1,2,\ldots, t$, and $\tau_i \notin <\tau_j>$, for all $i\neq j$.
\end{lemma}

\begin{proof}
Let $\beta^{-1}=(a_1~a_2 \cdots a_r)$ and $\alpha^{-1} = (b_1~b_2 \cdots b_s)$ be two disjoint cycles, where $r$ and $s$ are even. Note that $(\alpha \beta)^{-1}= \beta^{-1}\alpha^{-1} \in A_n$ can be written as
\begin{align*}
\beta^{-1} \alpha^{-1} &=   (a_1 ~a_2)  (b_1~b_2) ~ (a_2~a_3 \cdots a_{r}) (b_2~b_3 \cdots b_s) \\
&=   (b_2~b_1 ~x)  (b_1~a_2~x)(a_2~a_1~ x)(a_1~b_1~x) ~ (a_2~a_3 \cdots a_{r}) (b_2~b_3 \cdots b_s).
\end{align*}

Note that the last two cycles in the product are odd, and thus we can write them as a product of $3$-cycles, as indicated in Lemma \ref{lem3cycles1}. It is easy to see now that $(\alpha \beta)^{-1}$ can be written as claimed.
\end{proof}

The previous two lemmas indicate that we should only need one extra person to undo whatever chaos the  $3$-brain-swapping machine might have created. We make this result explicit in the following theorem.

\begin{theorem}\label{thm3cycles}
Let $\sigma \in A_n$, $n>2$. Then, $\sigma^{-1}$ can be written as a product of $3$-cycles
\[
\sigma^{-1} = \tau_1 \tau_2 \cdots \tau_t,
\]
where $\tau_i \in A_{n+1} \setminus A_n$,  for all $i=1,2,\ldots, t$, and $\tau_i \notin <\tau_j>$, for all $i\neq j$.
\end{theorem}

\begin{proof}
Given $\sigma \in A_n$, we know we can write $\sigma^{-1}$ as a product of disjoint cycles. Moreover, after rearrangement (if needed) we can put all the appearing even cycles (if any) in pairs. Hence, Lemmas \ref{lem3cycles1} and \ref{lem3cycles2}  guarantee the desired factorization of $\sigma^{-1}$.
\end{proof}

\subsection{Products of $p$-cycles, $p>5$} 

In order to give an answer to what would happen if the machine now swapped $p$ brains cyclically, where $p$ is a prime larger than three, we need to set some notation first. 

Fix $n, p \in \N$, where $p>3$ is prime and $n>2$.
Then, we define the following ordered lists of numbers
\begin{align*}
 [x]_{n,p} &= x_1~x_2~\ldots x_{p-3}\\
 [x]_{n,p}^{-1} &= x_{p-3}~\ldots x_2~x_1
\end{align*}
where $x_i \gg n$, for all $i= 1, 2, \ldots, p-3$. We will use these lists to write cycles. For example $(1~2~[x]_{3,7}) = (1~2~x_1~x_2~x_3~x_4)$, where $x_1, x_2, x_3, x_4 \gg 3$.

\begin{theorem}\label{thmpcycles}
Let $n\in \N$ and $p$ be a prime such that $p>3$ and $n>2$. If $\sigma \in A_n$, then, $\sigma^{-1}$ can be written as a product of $p$-cycles
\[
\sigma^{-1} = \tau_1 \tau_2 \cdots \tau_t,
\]
where $\tau_i \in A_{n+(p-3)} \setminus A_n$,  for all $i=1,2,\ldots, t$, and $\tau_i \notin <\tau_j>$, for all $i\neq j$.
\end{theorem}

\begin{proof}
We will first prove that every element $\sigma \in A_n$ can be written as a product of $p$-cycles, then we will prove that the factors used satisfy the conditions claimed in the theorem.

We know that every element in $A_n$ can be written as a product of odd cycles and an even number of even cycles. We will first deal with odd cycles. For $k=3$ we can write
\[
(1~2~3) = (1~3~2~[x]_{n,p}^{-1})(2~1~3~[x]_{n,p}),
\]
and so we need $p-3$ new elements to write a $3$-cycle as a product of $p$-cycles. For $k>3$ we know that every $k$-cycle in $A_n$ can be written as a product of $\frac{k-1}{2}$ distinct $3$-cycles as follows:
\begin{equation}\label{eqnfactor3cycles}
(a_1~a_2 \cdots a_k) = (a_1~a_2 ~a_3) (a_3~a_4 ~a_5) \cdots (a_{k-2}~a_{k-1}~a_k).
\end{equation}

So, if $\sigma \in A_n$ is a $k$-cycle with $k>3$, then we write it as in Equation (\ref{eqnfactor3cycles}). After that we write each of the $3$-cycles in that product as a product of two $p$-cycles, as discussed in the case $k=3$ above. Hence, we can write $\sigma$ as a product of $p$-cycles by incorporating $p-3$ new elements.

For products of two even cycles we let $\sigma=(a_1~a_2 \cdots a_r)$ and $\tau = (b_1~b_2 \cdots b_s)$, where $r$ and $s$ are even. Note that $\sigma \tau \in A_n$ can be written as
\begin{align*}
\sigma \tau &=   (a_1 ~a_2)  (b_1~b_2) ~ (a_2~a_3 \cdots a_{r}) (b_2~b_3 \cdots b_s) \\
&=   (b_2~b_1 ~a_2~[x]_{n,p}^{-1})  (a_2~a_1~b_1~[x]_{n,p}) ~ (a_2~a_3 \cdots a_{r}) (b_2~b_3 \cdots b_s) 
\end{align*}
where the first two cycles need $p-3$ new elements to have length $p$, and all the other cycles involved are odd, and thus can be written as products of $p$-cycles.

In order to see that all the $p$-cycles obtained in the final product describing $\sigma^{-1}$ generate distinct subgroups of $A_{n+(p-3)}$ we notice that, in the majority of the cases, a pair of these $p$-cycles will not  `move' the same $p$ elements, and thus the groups they generate would intersect trivially. Thus, the only case that should be analyzed is when we get two cycles of the form $\tau_1=(1~3~2~[x]_{n,p}^{-1})$ and $\tau_2=(2~1~3~[x]_{n,p})$. In this case, if we assumed that these elements generate the same subgroups then $\tau^i_2=\tau_1$, for some $i$. However, if we look at $\tau^i_2(2)$ we can see that the only possibility for $\tau^i_2=\tau_1$ to happen would be when $i=p-1$, but it is easy to see that $\tau^{p-1}_2=\tau_2^{-1} \neq \tau_1$. This finishes the proof.
\end{proof}

We conclude that as long as the brain-swapping machine Professor Farnsworth and Amy build swaps a prime number of brains cyclically, they can always fix the chaos created by incorporating enough extra characters to the mix. 



\begin{thebibliography}{99}
\bibitem{Gz} Casey Chan, Futurama writer invented a new math theorem just to use in the show (2010)  \verb+http://gizmodo.com/5618502/futurama-writer-invented-a-new-math-theorem-just-to-use-in-the-show+.

\bibitem{EHN} Ron Evans, Lihua Huang, Tuan Nguyen, Keeler's theorem and products of distinct transpositions, \emph{Amer. Math. Monthly} \textbf{121} (2014) 
136--144. 

\bibitem{FT} Hristo Georgiev, The Futurama theorem explained. \emph{The Commutator} \textbf{2} (2010) 18--20.

\bibitem{AMS} Tony Phillips, Math in the Media, Amer. Math. Soc., Original math on Futurama (2010), \verb+http://www.ams.org/news/math-in-the-media/10-2010-media+.

\bibitem{WGA} Previous Nominees \& Winners of the Writers Guild Awards (last accessed on 08/15/16). \verb+http://awards.wga.org/wga-awards/previous-nominees-winners+.

\bibitem{SS}  Simon Singh,  \emph{The Simpsons and Their Mathematical Secrets}. Bloomsbury, NY, 2013.

\bibitem{IS} The prisoner of Benda.  The Infosphere, the Futurama Wiki (last accessed on 08/15/16).  \verb+http://theinfosphere.org/The_Prisoner_of_Benda+.
\end{thebibliography}
\end{document}